\documentclass[12pt,a4paper]{amsart} \usepackage[english]{babel}
\usepackage{graphicx} \usepackage{mathtools} \usepackage{amsmath}
\usepackage{amssymb} \usepackage{amsthm}
\usepackage{physics} 
\usepackage{color}
\usepackage{float}
\usepackage{indentfirst}
\usepackage{comment}
\usepackage{url,orcidlink}
\usepackage{listings}
\usepackage{csquotes}
\theoremstyle{plain}
\newtheorem{theorem}{Theorem}
\newtheorem*{theorem*}{Theorem}
\newtheorem{prop}[theorem]{Proposition}
\newtheorem{lemma}[theorem]{Lemma}
\newtheorem{cor}[theorem]{Corollary}

\newtheorem*{MainTheorem1}{Main Theorem 1}
\newtheorem*{MainTheorem2}{Main Theorem 2}

\theoremstyle{remark}
\newtheorem{remark}[theorem]{Remark}

\theoremstyle{definition}

\def\AA{\mathbb{A}}

\def\FF{\mathbb{F}}
\def\NN{\mathbb{N}}
\def\QQ{\mathbb{Q}}

\def\ZZ{\mathbb{Z}}

\def\hh{\mathcal{H}}
\def\ll{\mathcal{L}}

\newcommand{\N}{\operatorname{N}}

\title[Metaconjugation and Quadratic Forms]{Metaconjugation and
  Quadratic Forms}

\author[Cardoso]{Adriana Cardoso\,\orcidlink{0000-0003-0669-1191}}
\address{CMUP --- Centro de Matemática da Universidade do Porto,
  Faculdade de Ci\^encias da Universidade do Porto, Rua do Campo
  Alegre, 4169-007 Porto, Portugal}
\email{adriana.cardoso,antonio.leite,ajmachia@fc.up.pt}
\author[Leite]{António Leite\,\orcidlink{0009-0001-9519-1350}}
\author[Machiavelo]{Ant\'onio Machiavelo\,\orcidlink{0000-0002-7595-7275}}

\author[Moreira]{Rafael Moreira\,\orcidlink{0009-0005-3786-3410}}

\author[Ro\c cadas]{Lu\'is Ro\c cadas\,\orcidlink{0000-0001-5579-5734}}
\address{Centro de Matemática, Universidade do Minho --- Polo
  CMAT-UTAD Universidade de Tr\'as-os-Montes e Alto Douro, Quinta de
  Prados, 5001-801 Vila Real, Portugal}
\email{rocadas@utad.pt}

\date{\today}
\begin{document}

\begin{abstract}
  We present a quaternionic proof that the quadratic form
  $t^2+2x^2+5y^2+10z^2$ represents all positive integers, and that the
  form $t^2+x^2+7y^2+7z^2$ represents all natural numbers which are
  neither equal to $3 \cdot 7^\ell$ nor equal to $6 \cdot 7^\ell$, for
  any $\ell \in \NN_0$. To do this we introduce a technique that may
  be useful for other purposes, and that we call
  \textit{metaconjugation}.
\end{abstract}


\maketitle
\section{Introduction}

The problem of determining the positive integers that are represented
by a quadratic form has quite a long history. In 1770, Lagrange proved
the \textit{four-square theorem}, which states that the form
$t^2+x^2+y^2+z^2$ is universal, that is, represents all natural
numbers. In 1917, Ramanujan \cite{Ramanujan} investigated forms of the
shape $at^2+bx^2+cy^2+dz^2$, where $a \leq b \leq c \leq d$ are
positive integers, and concluded that there are at most 55 such
universal forms. Later, Dickson determined that 54 of Ramanujan's
quadratic forms were indeed universal. Louis J. Mordell, in \S{4} of
Chapter~19 of his book on Diophantine equations, deals with the
problem of the representation of numbers by the quaternary forms
$t^2+bcx^2+cay^2+abz^2$, where $a,b,c$ are non-zero integers, by
describing a method that works for $c=1$ and some specific values of
$a$ and $b$ (see \cite[pp.~169--173]{MordellBook}).

The above results were all derived using the theory of quadratic
forms. In 1919, Hurwitz \cite{Hurwitz} was able to give an alternative
proof of the four-square theorem, using quaternions, in particular a
quaternion order whose elements are now known as the \textit{Hurwitz
  integers}. Along this line, in 2008 Deutsch \cite{Deutsch} proved
that eight of Ramanujan's quadratic forms are universal through a
quaternionic argument, using the fact that these forms are norm forms
of quaternion orders contained in Euclidean rings, but his strategy
failed to work for the form $t^2+2x^2+5y^2+10z^2$, despite being the
norm form of an order. Later, following Deutsch's work, Fitzgerald
\cite{Fitzgerald} found a basis for a Euclidean ring that contains the
order associated to this form. However, he did not manage to provide a
proof of the universality of this form using quaternions. Fitzgerald
was also able to give a basis for a maximal order contained in the
algebra $\left(\frac{-1,-7}{\QQ}\right)$ which is known to be a (left
and right) principal ideal domain (PID) but not Euclidean for the
norm. The form $t^2+x^2+7y^2+7z^2$ is the norm form of the order of
this algebra consisting of the quaternions with integral coefficients
for the standard basis, i.e.~$\ZZ[1,i,j,k]$, which we will call the
\textit{Lipschitz integers} of that order.

The purpose of this paper is to provide a quaternion approach for
finding the positive integers represented by the definite quadratic
quaternary forms $t^2+2x^2+5y^2+10z^2$ and $t^2+x^2+7y^2+7z^2$, which
we designate as \textbf{the 1-2-5-10 form} and \textbf{the 1-1-7-7
  form}, respectively.  For that, we use the arithmetic properties of
appropriate maximal orders, and our arguments are built upon a process
that we call \textbf{metaconjugation}, which we use to transform
elements from a maximal order into elements of the same norm that
belong to the Lipschitz integers of the respective
algebra. Metaconjugation refers here to the transformation of an
element $\gamma$, of a given order, into the element
$\sigma \gamma \tau^{-1}$ where $\sigma, \tau$ are in that order and
have the same norm. The idea is then to metaconjugate an element of a
maximal order containing all Lipschitz integers into a Lipschitz
integer. Using that we show the following two results.

\begin{MainTheorem1}
  The quadratic form $t^2+2x^2+5y^2+10z^2$ is universal, that is, it
  represents all $n \in \NN$.
\end{MainTheorem1}

\begin{MainTheorem2}
  The quadratic form $t^2+x^2+7y^2+7z^2$ represents exactly the
  natural numbers not of the form
  $n = 2^\varepsilon \cdot 3 \cdot 7^\ell$, for
  $\varepsilon \in \{0,1\}$ and $\ell \in \NN_0$.
\end{MainTheorem2}
\section{The general setting}\label{Sec:general_set}

Let $\AA:=\left(\frac{-a,-b}{\QQ}\right)$, where $a, b$ are non-zero
square-free integers, be the rational quaternion algebra with the
$\QQ$-basis $1, i, j,k$ and multiplication determined by $i^2=-a$,
$j^2=-b$, and $k=ij$. One wishes to determine the natural numbers $n$
for which there is an element $\alpha = t+x i+y j+z k$ in the order of
the Lipschitz quaternions of $\AA$, $\ll_{a,b}:=\ZZ[1,i,j,k]$, such
that $\N(\alpha)=n$, where
$\N(\alpha) = \alpha\bar{\alpha} = t^2+ax^2+by^2+abz^2$ is the
\emph{norm} of $\alpha$. Here $\bar{\alpha} = t-x i-y j-z k$ is the
\textit{conjugate} of $\alpha$. Because conjugation is an involution,
an order $\hh$ of $\AA$ is a left PID if and only if is a right PID,
and we thus may just call it a PID and forget about the adjectives
left and right.  Observe that the order $\ll_{a,b}$ can never be a
PID, as otherwise it would be maximal (cf. \cite[Prop.~1]{OrdersPID}),
and one can easily prove this is not true. To tackle the problem of
the representation of positive integers by the just mentioned norm
form, in the two cases under consideration in this paper, we will make
use of a maximal order $\hh_{a,b}$ that contains $\ll_{a,b}$, and is a
PID.

Recall that an element $\pi$ of an order $\hh$ of $\AA$ is said to be
\emph{prime} if $\N(\pi)=p$, a rational prime. In this case, one says
that $\pi$ is a prime of $\hh$ \emph{above} $p$. An element of $\hh$
is said to be \textit{primitive} if it is not divisible, in $\hh$, by
any rational prime. If $\hh$ is a PID, then given any primitive
element $\alpha\in\hh$, and any factorization of its norm into a
product of rational primes, $\N(\alpha)=p_1 p_2\cdots p_r$ (here order
matters), there is a factorization $\alpha =\pi_1 \pi_2\cdots \pi_r$
such that $\pi_t$ is a prime above $p_t$, for {\small
  $t=1,2,\ldots,r$}, and such a factorization is unique up to
\emph{unit-migration} (see \cite[Theorem 11.4.8]{Voight} and
\cite[Theorem 2]{OrdersPID}). This means that any other factorization
that is \emph{modelled} on the same factorization of $N(\alpha)$
(i.e.~the order of the $p_t$ is the same) must have the form
\[\alpha = (\pi_1 u_1) (u^{-1}_1 \pi_2 u_2) 
  \cdots (u^{-1}_{r-1} \pi_{r-1} u_r) (u^{-1}_r \pi_r),\]
where $u_t\in\hh^*$, the group of units of $\hh$, for all {\small
  $t=1,2,\ldots,r$}.

If $\gamma\in\hh$ is such that $p\mid \gamma$, then
$p\mid \Tr(\gamma)=\gamma+\bar{\gamma}$ and $p^2\mid\N(\gamma)$
(recall that $\Tr(\gamma)$ and $\N(\gamma)$ are integers --- see
Corollary~10.3.6 and Lemma~10.3.7 in \cite{Voight}). In particular, if
$i, j\in\hh$, as $\N(i)=a$ and $\N(j)=b$ are both square-free, we see
that $i$ and $j$ are primitive in any order containing $\ll_{a,b}$.
The next lemma guarantees the existence, for all primes $p$ and in any
order $\hh$ containing $\ll_{a,b}$, of a primitive quaternion with
norm equal to a multiple of $p$. We will then point out that it
follows from this that, when $\hh$ is a PID containing $\ll_{a,b}$,
the norm form for $\hh$ is universal.

\begin{lemma} Let $a,b\in\ZZ$ be square-free integers, and $\ll_{a,b}$
  as above. Let $\hh$ be an order of $\AA$ containing $\ll_{a,b}$.
  Then, given any rational prime $p$, there exists
  $\gamma \in \ll_{a,b}$ which is primitive in $\hh$, and such that
  $p \mid \N(\gamma)$. Moreover, if $p$ does not divide $ab$, then we
  may take $\gamma = 1 + ri + sj$, for some $r,s \in \ZZ $ with
  $0 \leq r,s < p$.
\end{lemma}

\begin{proof}
  If $p \mid ab$, take $\gamma = i$ if $p\mid a$, and take
  $\gamma = j$ if $p\mid b$. For $p = 2$, if either $a$ or $b$ is
  even, we are done.  Now, if $a$ and $b$ are both odd, then
  $\N(1+i)=1+a$ and $\N(1+j) = 1+b$ are both even. As
  $\Tr(1+i)=\Tr(1+j)=2$, the only way for both $1+i$ and $1+j$ not to
  be primitive is if $\frac{1+i}2, \frac{1+j}2\in \hh$. But then
  $\frac{1+i}2 \frac{1+j}2 = \frac{1+i+j+k}4$ would be in $\hh$, which
  is not true, as its trace is $\frac12$. Hence, either $\gamma = 1+i$
  or $\gamma = 1+j$ is a primitive element with even norm.

  Let, then, $p$ be an odd prime that does not divide $ab$. Then the
  sets
  \[ \left \{ 1 + ar^2 \colon r \in \FF_p \right\}, \; \left \{ -bs^2
      \colon s \in \FF_p \right\}\]
  each have $\frac{p+1}{2}$ distinct elements of $\FF_p$. It follows
  that their intersection is not empty. Hence, there will exist
  $r,s \in \FF_p$ such that
  \[ 1 + ar^2 \equiv -b s^2 \pmod{p}.\]
  Taking $\gamma = 1+ri+sj$, we get that $p \mid \N(\gamma)$. As
  $\Tr(\gamma)= 2$, either $\gamma$ is primitive, or
  $\frac{\gamma}2\in \hh$ is primitive and its norm is still divisible
  by $p$.
\end{proof}

\begin{theorem}\label{PID_form_universal}
  Let $\hh$ be a PID order in $\AA$ containing $\ll_{a,b}$. Then the
  norm form associated with $\hh$ is universal.
\end{theorem}

\begin{proof}
  Since the norm is multiplicative, we only need to show that there
  exists $\gamma \in \hh$ with norm equal to $p$, for all primes $p$.
  By the previous lemma, there is a primitive
  $\gamma \in \ll_{a,b} \subseteq \hh$ with norm divisible by
  $p$. Since $\hh$ is a PID, by \cite[Thm.~1]{OrdersPID} $\gamma$ has
  a right divisor with norm $p$.
\end{proof} 

In the next sections we will show how to use maximal orders
$\hh_{a,b}$ containing $\ll_{a,b}$, for $(a,b)=(2,5)$ and
$(a,b)=(1,7)$, that happen to be PIDs, and hence their norm forms are
universal, to determine the positive integers represented by the norm
form of $\ll_{a,b}$. It is in this context that we will use what we
call \textit{metaconjugation}. This consists of taking, for
$\gamma \in \hh_{a,b}$, appropriate quaternions
$\sigma, \tau \in \hh_{a,b}$ with the same norm, such that
$\sigma \gamma \tau^{-1}$, which still has the same norm as $\gamma$,
lands in $\ll_{a,b}$. That is, in order to prove that $n \in \NN$ is
represented by the norm form of $\ll_{a,b}$, it is enough to find, for
$\gamma \in \hh_{a,b}$ with norm $n$, quaternions
$\sigma, \tau \in \hh_{a,b}$ with $\N(\sigma) = \N(\tau)$, and such
that $\sigma \gamma \tau^{-1} \in \ll_{a,b}$. As the norm is
multiplicative, we may consider only $\gamma$ of prime norm.

Observe that if $\hh_{a,b} = \ZZ[v_0,v_1,v_2,v_3]$, and
$\sigma \gamma \bar{\tau} = t v_0 + x v_1 + y v_2 + z v_3$, then one
has that $\sigma \gamma \tau^{-1} \in \hh_{a,b}$ if and only if
$\N(\tau)$ divides $t,x,y,z$, since
\[ \sigma \gamma \tau^{-1} = \frac{1}{\N(\tau)} \sigma \gamma
  \bar{\tau}.\] Also, note that we may consider $\sigma, \tau$ up to
left associates in $\ll_{a,b}$, since
\[\sigma \gamma \tau^{-1} \in \ll_{a,b} \iff (u\sigma) \gamma
  \tau^{-1} \in \ll_{a,b} \iff \sigma \gamma (v\tau)^{-1} \in
  \ll_{a,b}, \]
for all $u,v \in \ll_{a,b}^*$.
  
We treat each of the forms 1-2-5-10 and 1-1-7-7 separately.

\section{The form 1-2-5-10}\label{Sec:1-2-5-10}

The form $t^2+2x^2+5y^2+10z^2$, when considered over $\QQ$, is the
norm form of the quaternion algebra $\left (\frac{-2,-5}{\QQ}\right)$,
and the problem of finding if it represents a given $n\in\NN$ is the
same as finding if there is an element of norm $n$ in the order of
Lipschitz quaternions $\ll_{2,5}:=\ZZ[1,i,j,k]$. To deal with this
question using the approach sketched in the previous section, we use
the maximal order $\hh_{2,5} := \ZZ [v_0, v_1, v_2, v_3]$, where
\begin{alignat*}{3}
	& v_0 = 1,  & v_1= \frac{2+i-k}{4},  \\
	& v_2 = \frac{2+3i+k}{4}, \quad  & v_3 = \frac{1+i+j}{2}.
\end{alignat*}
This order contains $\ll_{2,5}$ and is Euclidean for the norm (see
\cite[Theorem 4]{Fitzgerald}, and thus a PID. Note that the norm form
for this order is
\begin{multline*}
  \N( a + b v_1 + c v_2 + d v_3 ) = a^2+b^2+2c^2+2d^2+ab+ac+bd+2cd = \\
  =\left( a + \frac{b+c+d}{2} \right)^2 + 2\left( \frac{b+3c+2d}{4}
  \right)^2 + 5\left( \frac{d}{2} \right)^2 + 10\left( \frac{c-b}{4}
  \right)^2.
\end{multline*}
By Theorem~\ref{PID_form_universal}, this norm form is universal.  The
following lemma describes which elements of $\hh_{2,5}$ are in
$\ll_{2,5}$ in terms of their coordinates in the basis given above.

\begin{lemma}\label{L5_prop}
  A quaternion $\alpha = a + b v_1 + c v_2 + d v_3$ of $\hh_{2,5}$ is
  in $\ll_{2,5}$ exactly when $d$ is even and $b \equiv c \pmod{4}$.
\end{lemma}

\begin{proof}
  When expressed in the coordinates $1,i,j,k$, one has
  \[\alpha = \left(a + \frac{b+c}2+\frac{d}2\right) + 
    \left(\frac{b+3c}4+\frac{d}2 \right) i +\left(\frac{d}2\right)j +
    \left(\frac{c-b}4\right) k,\] and it is very easy to see that
  $\alpha\in\ll_{a,b}$ if and only if the last two of these
  coordinates are integers.
\end{proof}

The unit groups of both $\hh_{2,5}$ and $\ll_{2,5}$ are as follows.

\begin{lemma}
  $\ll_{2,5}^* =\{\pm 1 \}$ and
  $\hh_{2,5}^* = \{\pm 1, \pm v_1, \pm (1-v_1)\}$.
\end{lemma}

In the next proposition, we will show that our metaconjugation process
succeeds in finding a quaternion $\ll_{2,5}$ with norm $n$, for all
$n \in \NN$. As described before, it is enough to show that, for any
$\gamma \in \hh_{2,5}$, there are $\sigma, \tau \in \hh_{2,5}$ of
equal norm such that $\sigma \gamma \tau^{-1} \in \ll_{2,5}$. It turns
out that it is enough to consider quaternions of norm $4$.

\begin{prop}\label{H5_to_L5}
  For any $\gamma \in \hh_{2,5}$ there are
  $\sigma, \tau \in \hh_{2,5}$, with $\N(\sigma) = \N(\tau) = 4$, such
  that $\sigma \gamma \tau^{-1} \in \ll_{2,5}$ .
\end{prop}

\begin{proof}
  By the general discussion in Section~\ref{Sec:general_set}, if
  $\N(\sigma) = \N(\tau) = 4$ and
  $\sigma \gamma \bar{\tau} = t + xv_1 + y v_2 + z v_3$, then one has
  that $\sigma \gamma \tau^{-1} \in \hh_{2,5}$ if and only if $4$
  divides $t,x,y,z$.  Additionally,
  $\sigma \gamma \tau^{-1}\in\ll_{2,5}$ if and only $2\mid \frac{z}4$
  and $\frac{x}4\equiv \frac{y}4\pmod{4}$. Thus,
  \[
  \sigma \gamma \tau^{-1} \in \ll_{2,5} \iff \begin{cases}
  	t, x, y \equiv 0 \pmod{4}, \\
  	z \equiv 0\pmod{8}, \\
  	x \equiv y \pmod{16}.
  \end{cases}
\]

It follows, in particular, that if the coefficients
$\sigma \gamma \bar{\tau}$ are divisible by $16$, then
$\sigma \gamma \tau^{-1} \in \ll_{2,5}$. Thus, if we write
$\gamma = 16 \beta + \Tilde{\gamma}$, for some
$\beta, \Tilde{\gamma} \in \hh_{2,5}$ such that the coefficients of
$\Tilde{\gamma}$ are in $\{0,1, \dots, 15\}$, we have that
\[ \sigma \gamma \tau^{-1} = 4 \sigma \beta \bar{\tau} + \sigma
  \Tilde{\gamma} \tau^{-1}, \] from which follows that
$\sigma \gamma \tau^{-1} \in \ll_{2,5} \iff \sigma \Tilde{\gamma}
\tau^{-1} \in \ll_{2,5}$, as $4\sigma \beta \bar{\tau} \in
\ll_{2,5}$. Of course, we can consider $\sigma, \tau$ only up to sign.
  
Hence, we are reduced to check that, for each
$\Tilde{\gamma} \in \hh_{2,5}$ with coefficients in
$\{0,1, \dots, 15\}$, there always exist quaternions
$\sigma, \tau \in \hh_{2,5}$ of norm 4 such that
$\sigma \Tilde{\gamma}\tau^{-1} \in \ll_{2,5}$. Using the PARI/GP
\cite{PARI} code presented in the next section, we have verified
computationally our claim. The code was run in a laptop of 16GB of
RAM, and it took around 2 seconds to check all possibilities for
$\Tilde{\gamma}$.
\end{proof}

The previous result yields that, if there exists
$\gamma \in \hh_{2,5}$ of norm $n$, then it also exists
$\alpha \in \ll_{2,5}$ of norm $n$. Since the norm form of $\hh_{2,5}$
is universal, we get our first main result:

\begin{MainTheorem1}
  The form \emph{1-2-5-10} is universal.
\end{MainTheorem1}

\section{The form 1-1-7-7}\label{Sec:1-1-7-7}

We now consider the form $t^2+x^2+7y^2+7z^2$, which is the norm form
of the order $\ll_{1,7}:=\ZZ[1,i,j,k]$. In this case we use the
maximal order
\[ \hh_{1,7} := \ZZ \left [1, i, \omega, i \omega \right],\]
where $\omega=\frac{1}{2}(1+j)$. It is well known that this maximal
order is a PID, but not Euclidean for the norm (see
\cite[p.~95]{Brzezinski} and \cite[p.~155]{Vigneras}). The norm form
of $\hh_{1,7}$ is
\[ \N( a + b i + c\omega + d i \omega )= a^2 + b^2 +2c^2 + 2d^2 + ac +
  bd,\]
and a simple calculation yields the following lemma.

\begin{lemma}
  The units of either $\ll_{1,7}$ or $\hh_{1,7}$ are $ \pm 1, \pm i$.
\end{lemma}

An arbitrary element $a + b i + c\omega + d i \omega$ of $\hh_{1,7}$
can be written in the canonical basis as
$ \left( a+\frac{c}{2} \right) + \left(b+\frac{d}{2} \right) i +
\frac{c}{2}j + \frac{d}{2}k$. Thus, it is clear that the order
$\ll_{1,7}$ is strictly contained in $\hh_{1,7}$, and that the
following holds.

\begin{lemma}\label{L7_prop} A quaternion
  $\alpha = a + bi + c \omega + d i \omega \in \hh_{1,7}$ is in
  $\ll_{1,7}$ if and only if $c\equiv d \equiv 0 \pmod{2}$.
\end{lemma}

From this one can deduce the following result.

\begin{prop}\label{diff_parity}
  If
  $\alpha = a + bi + c \omega + d i \omega \in \hh_{1,7} \setminus
  \ll_{1,7}$ and has odd norm, then $c\not \equiv d\pmod{2}$.
\end{prop}
\begin{proof}
  Let
  $\alpha = a + bi + c \omega + d i \omega \in \hh_{1,7} \setminus
  \ll_{1,7} $ with $\N(\alpha)\equiv 1\pmod{2}$. By the previous
  Lemma, we have that $c$ and $d$ are not both even. Suppose that they
  are both odd. It follows that
  \[
    \N(\alpha) = a^2 + b^2 + 2c^2 + 2d^2 +ac+bd \equiv a^2 + b^2 + a +
    b \equiv 0 \pmod{2},
  \]
  which contradicts the hypothesis.
\end{proof}

\begin{remark}\label{d_odd}
  Suppose one has
  $\alpha = a + bi + c \omega + d i \omega\in \hh_{1,7} \setminus
  \ll_{1,7}$ of a given norm. As
  $i \alpha = -b + ai - d \omega + c i \omega$, by replacing $\alpha$
  by $i\alpha$ if necessary, we may assume that $\alpha$ is such that
  $c$ is odd and $d$ is even.
\end{remark}

In order to describe exactly which natural numbers are represented by
the norm form of the order $\ll_{1,7}$, we start by observing the
following.
\begin{lemma}\label{7n_to_n}
  Suppose $n$ is a positive integer for which $7n$ can be represented
  by the quadratic form \emph{1-1-7-7}. Then $n$ can also be
  represented by this form.
\end{lemma}

\begin{proof}
  Suppose there are $t,x,y,z \in \ZZ$ such that
  $t^2 + x^2 + 7y^2+7 z^2 = 7n$. Then, $t^2 + x^2 \equiv 0
  \pmod{7}$. As the only squares modulo $7$ are $0,1,2,4$, the only
  possibility is to have $x \equiv t \equiv 0 \pmod{7}$. But then
  \[
    7 \left( \frac{t}{7} \right)^2 + 7 \left(\dfrac{x}{7} \right)^2 +
    y^2 + z^2 = n.
  \]
\end{proof}

From this we can immediately extract the following conclusion.

\begin{prop}\label{Necessary}
  The form \emph{1-1-7-7} does not represent numbers of the form
  $2^\varepsilon \cdot 3 \cdot 7^\ell$ with $\varepsilon \in \{0,1\}$
  and $\ell \in \NN_0$.
\end{prop}

\begin{proof}
  This follows from the fact that neither 3 nor 6 are represented by
  1-1-7-7 and Lemma~\ref{7n_to_n}.
\end{proof}

To see that all the other numbers are representable by the form in
question, we need to show that if $n\in\NN$ is not of the above form,
then there exists $\alpha \in \ll_{1,7}$ such that $\N(\alpha) =
n$. Using the multiplicativity of the norm, we show in the next
proposition, as $1+i \in \ll_{1,7}$ has norm $2$, that in order to
prove our result it is enough to do it for all odd primes $p\neq 3$,
and for all numbers of the form $3p$, with $p\neq 2, 7$.

\begin{prop}\label{p+3p}
  If the form \emph{1-1-7-7} represents $p$ for all primes $p\neq 3$,
  and represents $3p$ for all odd primes $p\neq 7$, then it represents
  all $n\in\NN$ not equal $ 2^\varepsilon \cdot 3 \cdot 7^\ell$ for
  any $\varepsilon \in \{0,1\}$, $\ell \in \NN_0$.
\end{prop}

\begin{proof}
  Suppose that the form 1-1-7-7 represents $p$ for all primes
  $p\neq 3$, and represents $3p$ for all odd primes $p\neq 7$. Let $n$
  be not of the form $2^\varepsilon \cdot 3 \cdot 7^\ell$, for any
  $\varepsilon \in \{0,1\}$, $\ell \in \NN_0$.
  
  If $n=4m$ for some $m \in \NN$, then by
  Theorem~\ref{PID_form_universal}, there is $\alpha \in \hh_{1,7}$
  such that $\N(\alpha) = m$. In other words, there are
  $a,b,c,d \in \ZZ$ such that
  \[m = \left( a + \frac{c}{2} \right)^2 + \left( b + \frac{d}{2}
    \right)^2 + 7 \left( \frac{c}{2} \right)^2 + 7 \left( \frac{d}{2}
    \right)^2. \]
  But then
  \[ 4 m = \left( 2 a + c \right)^2 + \left( 2 b + d \right)^2 + 7 c^2
    + 7 d^2, \]
  and thus the form 1-1-7-7 represents $n$. 
  
  Now, as the norm is multiplicative, the hypothesis entails that any
  number that is not a multiple of $3$ is represented by the form
  1-1-7-7.  We are, thus, reduced to the case $n= 3^t m$ for some
  $t,m \in \NN$, with $t \geq 1$, $3 \nmid m$, and $4 \nmid m$. Let
  then $\sigma \in \ll_{1,7}$ be such that $m=\N(\sigma)$. If $t$ is
  even, $ n = 3^t m = \N(3^{t/2} \sigma)$. If $t$ is odd and greater
  than 1, then $3^t = 3^{2s+3} = 9^s \cdot 27$ for some $s\in \NN_0$,
  and one has
  \[n = 3^t m= 9^s \cdot 27\, m = \N(3^s (2+3i+j+k)\, \sigma). \]
  
  If $t=1$, then $n=3m$ with $m$ divisible by some prime $p\neq 2,
  7$. Hence, we have $n = 3p r$, for $r \in \NN$ such that
  $3 \nmid r$. By hypothesis, both $3p$ and $r$ are represented
  by 1-1-7-7, then so is $n$.
\end{proof}

We will now use the fact that the norm form of $\hh_{1,7}$ is
universal to prove that there is an element in $\ll_{1,7}$ of norm
equal to $p$, for any odd prime $p\neq 3$, as well as elements in
$\ll_{1,7}$ of norm $3p$ for odd primes $p\neq 7$.

Let us start by showing how to determine an element in $\ll_{1,7}$
with norm equal to a prime $p> 3$. Let
$\gamma = a + bi + c \omega + d i \omega$ be an element of $\hh_{1,7}$
with norm $p$. When $\gamma \in \ll_{1,7}$ we are done, and therefore
we assume that $\gamma \in \hh_{1,7} \setminus \ll_{1,7}$. It is now
enough to show that there always are $\sigma, \tau\in\hh_{1,7}$ with
the same norm and such that $\sigma \gamma \tau^{-1} \in
\ll_{1,7}$. It turns out that one can always choose quaternions
$\sigma, \tau$ of norm $6$ that do the job.

\begin{prop}\label{H7_to_L7}
  If $\gamma \in \hh_{1,7} \setminus \ll_{1,7}$ is such that
  $\N(\gamma)=p$ for a prime $p>3$, then there exist
  $\sigma, \tau \in \hh_{1,7}$ such that $\N(\sigma)=\N(\tau)=6$ and
  \[ \sigma \gamma \tau^{-1} \in \ll_{1,7}. \]
\end{prop}

\begin{proof}
  By Remark~\ref{d_odd}, we can assume that $c$ is odd and $d$ is
  even. Since the norm of $\gamma$ is odd and
  $\N(\gamma) = a^2 + b^2 + 2c^2 + 2 d^2 + ac + bd$, one gets that $b$
  must be odd. Therefore, we can assume, renaming the coefficients if
  necessary, that $\gamma$ can be written as
  \begin{equation}\label{forma_gamma}
    \gamma = a + (2b+1)i+(2c+1)\omega + 2d\, i\omega,
  \end{equation}
  where $a,b,c,d \in \ZZ$.
  
  By the discussion of Section~\ref{Sec:general_set}, we have that if
  $\sigma \gamma \bar{\tau} = t + xi + y \omega + z i \omega$, with
  $\sigma, \tau \in \hh_{1,7}$, and $\N(\sigma) =\N(\tau) = 6$, then
  one has that $\sigma \gamma \tau^{-1} \in \hh_{1,7}$ if and only if
  $6$ divides $t,x,y,z$. Additionally, if
  $\sigma \gamma \tau^{-1}\in \hh_{1,7}$, then it is in $\ll_{1,7}$ if
  and only if $2$ divides $y$ and $z$, i.e.~the coefficients of
  $\omega$ and $i\omega$. Hence,
  \[
    \sigma \gamma \tau^{-1} \in \ll_{1,7} \iff t,x \equiv 0 \pmod{6}
    \quad\text{and}\quad y,z \equiv 0\pmod{12}.
  \]

  From this it follows that when the coefficients of
  $\sigma \gamma \bar{\tau}$ are divisible by $12$ then
  $\sigma \gamma \tau^{-1} \in \ll_{1,7}$. If one writes
  $\gamma = 12\beta + \Tilde{\gamma}$ for some
  $\Tilde{\gamma}, \beta \in \hh_{1,7}$, with the coefficients of
  $\Tilde{\gamma}$ in \{0, 1, \dots, 11\}, we have
  \[\sigma \gamma \tau^{-1} = 2\sigma \beta \bar{\tau} + \sigma
    \Tilde{\gamma} \tau^{-1}.\] As
  $2\sigma \beta \bar{\tau} \in \ll_{1,7}$, we have that
  $\sigma \gamma \tau^{-1} \in \ll_{1,7}$ if and only if
  $\sigma \Tilde{\gamma} \tau^{-1} \in \ll_{1,7}$. Note that
  $\Tilde{\gamma}$ is still of the form of \eqref{forma_gamma}, and if
  $\N(\Tilde{\gamma}) \equiv 0 \pmod{3}$, then
  $\N(\gamma) \equiv 0 \pmod{3}$, since
  $\gamma = 12\beta + \Tilde{\gamma}$. However, as $\gamma$ is a prime
  above $p \neq 3$, this is impossible, so we can disregard the cases
  where $\Tilde{\gamma}$ has norm a multiple of $3$.

  Based on the above discussion, our goal reduces to check that, for
  each $\Tilde{\gamma}\in \hh_{1,7}$, with coefficients in
  $\{0, 1, \dots, 11\}$, there always exist quaternions
  $\sigma, \tau \in \hh_{1,7}$ of norm 6, up to left associates, such
  that $\sigma \Tilde{\gamma}\tau^{-1} \in \ll_{1,7}$. Using the
  PARI/GP \cite{PARI} code presented in the next section, we have
  verified computationally our claim. The code was run in a laptop of
  16GB of RAM, and it took around 2 seconds to check all possibilities
  of $\Tilde{\gamma}$.
\end{proof}

From the above discussion, we get the following corollary.

\begin{cor}\label{p_L7}
  Let $p\neq 3$ be a rational prime. Then $p$ is represented by
  \emph{1-1-7-7}.
\end{cor}

All that remains to be seen is that $3p$ is represented by 1-1-1-7 for
all odd primes $p\neq 7$. First, we show that, for an odd prime $p$,
the fact that $3p$ is represented by 1-1-7-7 is equivalent to the
existence of $\gamma \in \hh_{1,7} \setminus \ll_{1,7}$ with norm $p$.

\begin{prop}\label{H7_to_3p}
  The form \emph{1-1-7-7} represents $3p$ for an odd prime $p$ if and
  only if there exists $\gamma \in \hh_{1,7} \setminus \ll_{1,7}$ such
  that $\N(\gamma)=p$.
\end{prop}

\begin{proof}
  First, suppose that there exist $\alpha \in \ll_{1,7}$ such that
  $\N(\alpha)=3p$ for an odd prime $p$. Then, as
  $\ll_{1,7} \subset \hh_{1,7}$, by \cite[Thm.~2]{OrdersPID}, the
  quaternion $\alpha$ has a factorization modeled by $3p$, say
  $\sigma \gamma$, where $\N(\sigma)=3$ and $\N(\gamma)=p$, unique up
  to unit-migration. Suppose that $\gamma$ is in $\ll_{1,7}$, then, by
  Lemma~\ref{L7_prop}, it can be written as
  \[ \gamma = a + bi + 2c \omega + 2d i \omega,\] for some
  $a,b,c,d \in \ZZ$. From $ p = a^2 + b^2 + 8c^2 + 8d^2 + 2ac + 2bd$,
  one sees that $a$ and $b$ have different parity.
  
  Without loss of generality, we may assume that $\sigma$ is one of
  the representatives of the four right non-associate classes of
  primes above $3$.  Now,
  {\small \begin{alignat*}{1} (1-i+i\omega) \gamma & = (a-4d) +
      (b-a-4c)i + (b + 2c + 2d) \omega + (a + 2d)i\omega, \\
      (1+i\omega) \gamma & = (a -b -4d) + (b-4c)i + (b + 2c)\omega +
      (a + 2c + 2d) i \omega,\\
      (1+i-\omega) \gamma & = (a -b + 4c) + (a -4d) i - (a + 2d)\omega
      + (b+2c+2d)i \omega, \\
      (1+i-i\omega) \gamma & = (a + 4d) + (a+b+4c)i +(-b+2c-2d)\omega
      + (-a +2d) i \omega.
    \end{alignat*}}%
  By Lemma~\ref{L7_prop}, any one of these quaternions is in
  $\ll_{1,7}$ if and only if $a \equiv b \equiv 0\pmod{2}$, a
  contradiction. Hence, $\gamma$ must be in
  $\hh_{1,7} \setminus \ll_{1,7}$.
 
  Now, suppose there exists $\gamma \in \hh_{1,7} \setminus \ll_{1,7}$
  such that $\N(\gamma)=p$, for all odd primes $p$. As in the proof of
  Proposition~\ref{H7_to_L7}, we can write $\gamma$ as
  \[
    \gamma = a + (2b+1)i + (2c+1) \omega + 2d i \omega,
  \]
  for $a,b,c,d \in \ZZ$. Computing the product of $\gamma$, on the
  left, by $1-i\omega $ and by $1+i-\omega$, two quaternions above $3$,
  we obtain:
  {\footnotesize
    \begin{alignat*}{1}
      (1-i \omega) \gamma = (1+a+2b+4d) + (3+2b+4c)i - 2(b-c) \omega
      + (2(d-c)-a-1) i \omega,\\
      (1+i -\omega) \gamma = (1+a-2b+4c) + (a-4d)i - (a+2d) \omega +
      2(b+c+d+1) i \omega.
    \end{alignat*}
  }
  The first quaternion is in $\ll_{1,7}$ if and only if
  $a \equiv 1 \pmod{2}$, and the second one is in $\ll_{1,7}$ if and
  only if $a \equiv 0 \pmod 2$. Thus, depending the parity of $a$, we
  may choose $\sigma$ accordingly so that
  $\sigma \gamma \in \ll_{1,7}$, and we see that $3p$ is represented
  by 1-1-7-7.
\end{proof}

To prove that $3p$ is represented by 1-1-7-7, for $p\neq 7$, it is
thus enough to show that there is a prime in
$\hh_{1,7} \setminus \ll_{1,7}$ above $p$.

\begin{prop}\label{p_H7minusL7}
  There exists $\gamma \in \hh_{1,7} \setminus \ll_{1,7}$ such that
  $\N(\gamma) = p$ for all odd prime $p\neq 7$.
\end{prop}

\begin{proof}
  In the previous proof we provided two examples of elements in
  $\hh_{1,7} \setminus \ll_{1,7}$ of norm $3$. Thus, we may assume
  $p\neq 2, 3, 7$. Let $\gamma \in \ll_{1,7}$ be such that
  $\N(\gamma) = p$, whose existence is guaranteed by
  Corollary~\ref{p_L7}. By Lemma~\ref{L7_prop}, $\gamma$ can be
  written in the form $a + bi + 2c \omega + 2d i \omega$, for some
  $a,b,c,d \in \ZZ$. From $p = a^2 + b^2 + 8c^2 + 8d^2 + 2ac + 2bd$,
  one gets that $a$ and $b$ have different parity. As
  $i\gamma = -b + ai -2d \omega+ 2c i\omega$, we may assume $a$ odd
  and $b$ even, without loss of generality. Thus, renaming the
  coefficients if needed, we assume $\gamma$ can be written as
  \begin{equation}\label{gamma_L7_odd}
    \gamma = (1+2a) + 2bi + 2c \omega + 2d i \omega, 
  \end{equation}
  for some $a,b,c,d \in \ZZ$.
	
  Let us start by assuming that either the coefficient of $i$ or the
  coefficient of $i\omega$ is non-zero. Now, the conjugates of
  $\gamma$ by $\omega$ and by $\bar{\omega}$ are:
  \begin{alignat*}{1}
    \omega \gamma \omega^{-1} & = (1+2a) + (-b+2d)i +2c \omega -
    (b+2d) i \omega, \\
    \bar{\omega} \gamma \bar{\omega}^{-1} & = (1+2a) - (2b + 2d)i +
    2c\omega + (b - d)i \omega.
  \end{alignat*}
  Let $r=\nu_2(2b)$ and $s=\nu_2(2d)$, where $\nu_2$ denotes the
  2-adic valuation. Recall that one has
  $\nu_2(x+y) \geq \min \{ \nu_2(x),\nu_2(y)\}$, with equality if and
  only if $\nu_2(x) \neq \nu_2(y)$. Note that $r,s \geq 1$, and that
  at least one of them is smaller than $\infty$.
	
  We now consider the two possible cases: (1) $r \leq s $, and (2)
  $r > s $. In case (1), $\nu_2(-b+2d)=r-1=\nu_2(b+2d)$, and thus the
  2-adic valuation of the coefficients of $i$ and $i\omega$ in
  $\omega \gamma \omega^{-1}$ is strictly smaller than the 2-adic
  valuation of the same coefficients of $\gamma$. Thus, conjugating by
  $\omega^r$ leads to a quaternion whose coefficient of $i \omega$ is
  odd, and therefore it is in $\hh_{1,7} \setminus \ll_{1,7}$. In case
  (2), $\nu_2(2b+2d)=s$ and $\nu_2(b-d)=s-1$, and hence
  $\bar{\omega} \gamma \bar{\omega}^{-1}$ is such that the 2-adic
  valuation of its coefficient of $i$ is $s$ and of $i\omega$ is
  $s-1$, thus both have strictly smaller 2-adic valuations than the
  corresponding coefficients of $\gamma$. Hence, conjugating by
  $\bar{\omega}^s$ will produce a quaternion with odd coefficient of
  $i \omega$, and therefore in $\hh_{1,7} \setminus \ll_{1,7}$. In
  both cases, we have the desired result.
	
  Assume now that both the coefficients of $i$ and $i\omega$ are zero,
  i.e.~$\gamma$ is of the form $(1+2a) + 2c \omega$. Conjugating
  $\gamma$ by $1+i-\omega - i \omega$, a quaternion of norm $4$,
  yields:
  \[ (1+i-\omega - i \omega) \gamma (1+i-\omega - i \omega)^{-1} =
    (1+2a + c) - ci + 2c i\omega, \] and when $c$ is even this
  quaternion is of the form~\refeq{gamma_L7_odd}, and we are reduced
  to the previous case, as $c$ must not be not zero, since otherwise
  $\gamma$ would be an integer and its norm would not be
  prime. Finally, if $c$ is odd, then
  $i (1+i-\omega - i \omega) \gamma (1+i-\omega - i \omega)^{-1} = c+
  (1+2a + c)i - 2c \omega $ is of the form~\eqref{gamma_L7_odd}. If
  $1+2a= -c$, then we would have
  \[ \N(\gamma) = c^2 + 2(2c)^2 - (2c)c = 7c^2, \] which would
  contradict the assumption that $\N(\gamma)$ is a prime different
  from $7$. Therefore, $1+2a+c \neq 0$, and we are back to the case
  where at least one of the coefficients of $i$ and of $i\omega$ is
  non-zero. This completes the proof.
\end{proof}

The last two results imply the following corollary.

\begin{cor}\label{3p_L7}
  Let $p$ be an odd prime different from $7$. Then $3p$ is represented
  by \emph{1-1-7-7}.
\end{cor}

Finally, Corollaries~\ref{p_L7} and \ref{3p_L7}, along with
Proposition~\ref{p+3p}, prove the second main result of this paper.

\begin{MainTheorem2}
  The form \emph{1-1-7-7} represents $n \in \NN$ if and only if $n$ is
  not of equal to $2^\varepsilon \cdot 3 \cdot 7^\ell$, for any
  $\varepsilon \in \{0,1\}, \ell \in \NN_0$.
\end{MainTheorem2}

\section{PARI/GP Code}\label{Sec:PARI}

To make computations in a given quaternion order, we represent its
elements by matrices using the left regular representation, i.e.~the
element $\alpha = a v_0 + b v_1 + c v_2 + d v_3$ of a quaternion order
$\hh=\ZZ[v_0,v_1,v_2, v_3]$ is represented by the matrix, in the basis
$\{v_0,v_1,v_2, v_3\}$, of the linear map given by:
\begin{alignat*}{3}
  & \hh \ & \longrightarrow & \; \hh \\
  & h \ & \longmapsto & \; \alpha h
\end{alignat*}

For the orders dealt with above, $\hh_{2,5}$ and $\hh_{1,7}$, an
arbitrary element $\alpha$ is represented by the matrices
\footnotesize
\[ \begin{pmatrix}
	a & -b & -2c-2d & -b-2d \\
	b & a+b & b+2d & b-2c \\
	c & -d & a+b+c+d & b \\
	d & c+d & -b & a+c+d \\
      \end{pmatrix} \; ;\;
      \begin{pmatrix}
	a & -b-d & -2c & -2d \\
	b & a+c & -2d & 2c \\
	c & d & a+c & -b \\
	d & -c & b+d & a \\
\end{pmatrix}, \]
\normalsize
respectively.

Using this, we define the functions \textbf{quat5} and \textbf{quat7}
that take as input the coefficients of $\alpha$ in the basis of
$\hh_{2,5}$ and $\hh_{1,7}$, respectively, and output the appropriate
matrices. Next, we define the functions \textbf{q5conj},
\textbf{q5norm}, \textbf{q7conj} and \textbf{q7norm} that, given a
quaternion $q$ in matrix form, return the conjugate and the norm of
$q$ for the corresponding order. We also define the function
\textbf{quatbr}, that takes as input a quaternion in matrix form and
outputs the coefficients in vector form for better reading (hence the
\textbf{br}).

\footnotesize
\begin{lstlisting}
/* Representation of a H_5-quaternion as a matrix, its
conjugate, and its norm */

quat5(a,b,c,d)=[a,-b,-2*c-2*d,-b-2*d; b,a+b,b+2*d,b-2*c;
                         c,-d,a+b+c+d,b; d,c+d,-b,a+c+d];

q5conj(q)=quat5(q[1,1]+q[2,1]+q[3,1]+q[4,1],-q[2,1],
                                        -q[3,1],-q[4,1]);

q5norm(q)=(q*q5conj(q))[1,1];

/* Given any quaternion in matrix form, returns its
coordinates in the respective basis */

quatbr(q)=[q[1,1],q[2,1],q[3,1],q[4,1]]; 

/*----------------------------------------------------*/

/* Representation of a H_7-quaternion as a matrix, its
conjugate, and its norm */

quat7(a,b,c,d)=[a,-b-d,-2*c,-2*d;b,a+c,-2*d,2*c;
				  c,d,a+c,-b;d,-c,b+d,a];

q7conj(a)=quat7(a[1,1]+a[3,1],-a[2,1],-a[3,1],-a[4,1]);

q7norm(q)=(q*q7conj(q))[1,1];

/* Given any quaternion in matrix form, returns its
coordinates in the respective basis */

quatbr(q)=[q[1,1],q[2,1],q[3,1],q[4,1]]; 
\end{lstlisting}
\normalsize

As we wish to check if $\sigma \gamma \tau^{-1} \in \ll_{a,b}$, for
the orders in question, it is helpful to define functions
\textbf{isH5}, \textbf{isL5}, \textbf{isH7} and \textbf{isL7} that
check if a given quaternion, in matrix form, is in $\hh_{2,5}$,
$\ll_{2,5}$, $\hh_{1,7}$ or $\ll_{1,7}$, respectively. In order to
check $\alpha = a + b v_1 + c v_2 + d v_3$ is in $\ll_{a,b}$, we just
need to check if $\alpha \in \hh_{a,b}$, that is, if
$a,b,c,d \in \ZZ$, and if $\alpha$ satisfy the conditions of
Lemmas~\ref{L5_prop} and \ref{L7_prop}, respectively.

\footnotesize
\begin{lstlisting}
/* Returns True (1) if q is in H5 or in L5, respectively 
and False (0) otherwise. */

isH5(q)={
	for(i=1,4, if(type(q[i,1])!="t_INT", return(0)));
	return(1);
};

isL5(q)=(isH5(q) && (q[2,1]-q[3,1])%4==0 && q[4,1]%2==0);

/* Returns True (1) if q is in H7 or in L7, respectively 
and False (0) otherwise. */

isH7(q)={
	for(i=1,4, if(type(q[i,1])!="t_INT", return(0)));
	return(1);
};

isL7(q)=(isH7(q) && q[3,1]%2==0 && q[4,1]%2==0);
\end{lstlisting}
\normalsize

We also need to find all quaternions, up to left associates, of a
given norm $m$. For that purpose, we first define the functions
\textbf{repH5} and \textbf{repH7}, which take as input a positive
integer $m$ and returns all quaternions of norm $m$ in $\hh_{2,5}$ and
$\hh_{1,7}$, respectively. We will describe each case separately, as
they are different.

Firstly, consider the case of $\hh_{1,7}$, since it is simpler. Using
the previous notations, we want to find all quadruples
$(a,b,c,d) \in \ZZ^4$ such that
$\alpha= a + b i + c \omega + d i \omega$ has norm $m$. From
\begin{equation}\label{norm_alpha_pari_H7}
	4\N(\alpha) = (2a+c)^2 + (2b+d)^2 + 7 c^2 + 7d^2,
\end{equation}
it follows that
\[ \begin{cases}
    - \sqrt{\frac{4m}{7}} < c,d < \sqrt{\frac{4m}{7}}, \\
    - \frac{\sqrt{4m} - c}{2} < a < \frac{\sqrt{4m} - c}{2}, \\
    - \frac{\sqrt{4m} - d}{2} < b < \frac{\sqrt{4m} - d}{2},
\end{cases} \]
which give upper and lower bounds for $(a,b,c,d)$.  The function
\textbf{repH7} will go through all $a,b,c,d$ inside the above
boundaries and, whenever $\N(\alpha)=m$ holds, it adds the respective
quaternion, in matrix form, to a list, which it returns when all
possible quadruples are checked.

Secondly, we reduce the list of all quaternions of norm $m$ to a list
of all the left non-associate quaternions of that norm. In order to do
so, we define an auxiliary function \textbf{AssocL} that takes as
argument two quaternions $r1, r2$, in matrix form, and checks if they
are left associates. The function \textbf{QuatH7LAssoc} takes as input
a positive integer $n$ and returns a list of all the left
non-associate quaternions of that norm $n$. It does this by going
through the list given by \textbf{repH7}, using indexes $i$ and $j$,
and removing the quaternion in position $i$, if the quaternion in
positions $i$ and $j$ are left associate.

Now consider the case of $\hh_{2,5}$. As before, we want to find all
quadruples $(a,b,c,d)$ such that $\alpha= a + b v_1 + c v_2 + d v_3$
has norm $m$. From
\begin{equation}\label{norm_alpha_pari_H5}
  16\N(\alpha) = (4a+2b+2c+2d)^2 + 2 (b+3c+2d)^2 + 5 (2d)^2 + 10 (c-b)^2,
\end{equation}
it follows that
\[ \begin{cases}
    - \frac{\sqrt{\frac{16m}{5}}}{2} < d < \frac{\sqrt{\frac{16m}{5}}}{2}, \\
    - \sqrt{\frac{16m}{10}} < c- b < \sqrt{\frac{16m}{10}}, \\
    - \frac{16m}{2} - 2d < b + 3c < \frac{16m}{2}- 2d, \\
    - \frac{\sqrt{16m} - 2b-2c-2d }{4} < a < \frac{\sqrt{16m} -
      2b-2c-2d}{4}.
  \end{cases} \]
Setting $x= c-b, y=b+3c$, what we have are bounds for $a,x,y,d$. But
noting that $b = c- x, c = \frac{x+y}{4}$, it is clear that we also
have delimited the values of $a,b,c,d$. Using this, the function
\textbf{repH5} will go through all possible integral values of
$a,b,c,d$, and if $\N(\alpha)=m$, then it adds the respective
quaternion $\alpha$ to a list, which it returns at the end.

The function \textbf{PositiveQuater} takes as input $n\in\NN$ and
returns a list of all the quaternions of norm $n$ such that no two
quaternion in the list are associated in $\ll_{2,5}$. It does this by
going through list given by \textbf{repH5}, and simply removing a
quaternion if it is equal to $-\sigma$ for some $\sigma$ on the list,
as the only units of $\ll_{2,5}$ are $\pm1$.

\footnotesize
\begin{lstlisting}
/* List of all representations of m as a^2 + b^2 + 2c^2 +
2d^2 + ac + bd */

repH7(m)={
  local(n,a,b,c,d,L);
  L=List([]);
  n=4*m;
  for(d=-sqrtint(floor(n/7)),sqrtint(floor(n/7)),
    for(c=-sqrtint(floor(n/7)),sqrtint(floor(n/7)), 
      for(b=floor((-sqrtint(n)-d)/2),(sqrtint(n)-d)/2,
        for(a=floor((-sqrtint(n)-c)/2),(sqrtint(n)-c)/2,
          if(n==7*d^2+7*c^2+(2*b+d)^2+(2*a+c)^2,
            listput(~L,quat7(a,b,c,d)))))));
  return(L);
};

/* Checks if r1 and r2 are Left associates, returns True
(1) they are associated and False (0) otherwise.  */

AssocL(r1,r2)={
  if(r1==r2 || -r1==r2 || quat7(0,1,0,0)*r1==r2 || 
  quat7(0,-1,0,0)*r1==r2, return(1));
  return(0);
};

/* Return a list of all quaternions of norm n in H7, up
to left associates */

QuatH7LAssoc(n)={ 
  local(L,i,j);
  i=1;
  L=repH7(n);
  listsort(~L,1);
  while(i <= length(L), 
    j=i+1; 
    while(j <= length(L),            
      if(AssocL(L[i],L[j]), listpop(~L,i); j=i+1, j++));
    i++;
  );
  return(L)
};

/*-----------------------------------------------------*/

/* List of all representations of n as  a^2 + b^2 + 2*c^2
+ 2*d^2 + ab + ac + ad + bd + 2cd  */

repH5(m)={
  local(n,a,b,c,d,x,y,L);
  L=List([]);
  n=16*m;
  for(d=floor(-sqrtint(n/5)/2),sqrtint(n/5)/2,
    for(x=-sqrtint(n/10),sqrtint(n/10),
      for(y=-sqrtint(n/2)-2*d,sqrtint(n/2)-2*d,
        if((x+y)%4!=0, next());
        c=(x+y)/4;
        b=c-x;	
        for(a=floor((-sqrtint(n)-2*b-2*c-2*d)/4),
                              (sqrtint(n)-2*d-2*b-2*c)/4,
          if(n==(4*a+2*(b+c+d))^2+2*(b+3*c+2*d)^2+
                                    5*(2*d)^2+10*(c-b)^2,
            listput(~L,quat5(a,b,c,d)))))));
  return(L);
};

PositiveQuater(n)={ 
  local(L,i,j);
  i=1;
  L=repH5(n);
  listsort(~L,1);
  while( i <= length(L), 
    j=i+1; 
    while(j <= length(L),            
      if(L[i]==-L[j], 
        listpop(~L,i); j=i+1,
        j++));
    i++;
  );
  return(L)
};
\end{lstlisting}
\normalsize

We now present the code that does the verification mentioned at the
end of the proof of Proposition~\ref{H5_to_L5}. We define
\textbf{testMod16}, which tests if, for all $\Tilde{\gamma}$
quaternions of the form $a + b v_1 + c v_2 + dv_3$, with
$0 \leq a,b,c,d \leq 15$, there is a pair $(\sigma, \tau)$ of norm 4,
such that, $\sigma \Tilde{\gamma} \tau^{-1} \in \ll_{2,5}$. The
verification is done by the auxiliary function
\textbf{metaconjugation5}, that takes as input a quaternion $q$, in
matrix form, and a positive integer $n$, and checks if there exists
such pair $(\sigma, \tau)$ of quaternions of norm $n$.

\footnotesize
\begin{lstlisting}
/* Tests if, for all quaternion of the form a+b v1+c v2+
dv3, with coefficients between 0 and 15, there exists a
pair (sigma, tau) of quaternions of norm 4 such that 
sigma*q*tau^(-1) is in L5. If it does not find any, will 
print the respective coefficients [a,b,c,d] */ 

testMod16()={
  local(a,b,c,d,n);
  n=0;
  forvec(X=[[0,15],[0,15],[0,15],[0,15]], 
    [a,b,c,d]=X;
    q=quat5(a,b,c,d);
    if(!metaconjugation5(q,4), 
      print(X, " FAILS");
      return(0)));
  return(1);
};

/* For given quaternion q and positive integer n, checks
if there exists a pair (sigma, tau) of quaternions of
norm n such that sigma*q*tau^(-1) is in L5. If they
exists returns True (1), otherwise returns False (0) */

metaconjugation5(q,n)={
  local(L, aux);
  L=PositiveQuater(n);
  foreach(L, S, 
    foreach(L, T, 
      aux=S*q*(T^(-1));
        if(isL5(aux), return(1))));
  return(0);
};
\end{lstlisting}
\normalsize

Finally, we present the code used to verify what was claimed in last
part of the proof of Proposition~\ref{H7_to_L7}. We define
\textbf{testMod12}, which tests if, for all $\Tilde{\gamma}$
quaternions of the form
\[\Tilde{\gamma} = a + (2b+1)i +(2c+1)\omega + 2di\omega,\]
with $ 0 \leq a \leq 11, 0 \leq b,c,d \leq 5$ and having norm not
divisible by 3, there is pair $(\sigma, \tau)$ of quaternions of norm
6 such that $\sigma \Tilde{\gamma} \tau^{-1} \in \ll_{1,7}$. This
verification is done by the auxiliary function
\textbf{metaconjugation7}, that takes as input a quaternion $q$ and a
positive integer $n$, and checks if there exists such pair
$(\sigma, \tau)$ of quaternions of norm $n$.

\footnotesize
\begin{lstlisting}
/*Tests if, for all quaternion of the form a + (2*b+1)i+
(2*c+1)omega + (2*d) i omega, with coefficients between 0 
and 11, with norm not multiple of 3, there exists a pair 
(sigma, tau) of quaternions of norm 6 such that
sigma*q*tau^(-1)  is in L7. If it does not find any, will
print the respective coefficients [a,b,c,d]. */ 

testMod12()={
  local(a,b,c,d,q);
  forvec(X=[[0,11],[0,5],[0,5],[0,5]],
    [a,b,c,d]=X;
    q=quat7(a,2*b+1,2*c+1,2*d);
    if(q7norm(q)%3==0, next());
    if(!metaconjugation7(q,6), 
      print("FAILS FOR ", X);
	  return(0)));
  return(1);
};

/* For given quaternion q and a positive integer n,
checks if there  exists a pair (sigma, tau) of
quaternions of norm n such that  sigma *q* tau^(-1)
is in L7. If they exists returns True (1), otherwise
returns False (0) */

metaconjugation7(q,n)={
  local(L, aux);
  L=QuatH7LAssoc(n);
  foreach(L, S, 
    foreach(L, T, 
      aux=S*q*(T^(-1));
      if(isL7(aux), return(1))));
  return(0);
};
\end{lstlisting} 
\normalsize
\section*{Funding and Declaration}

The first three authors were partially supported by CMUP, member of
LASI, which is financed by national funds through FCT ---
\textit{Funda\c c\~ao para a Ci\^encia e a Tecnologia, I.P.}, under
the project with reference UID/00144/2025 and associated DOI given by
\url{https://doi.org/10.54499/UID/00144/2025}. The first and second
authors were also supported by the FCT doctoral scholarships with
references 2022.13732.BD
(\url{https://doi.org/10.54499/2022.13732.BD}), and 2025.01855.BD,
respectively. The fourth author was supported by a grant
``\textit{Novos Talentos}'', \#~327831, of the foundation
\textit{Calouste Gulbenkian}. The research of the fifth author was
partially financed by Portuguese Funds through FCT within the Project
UID/00013/2025 (\url{https://doi.org/10.54499/UID/00013/2025}).

The authors have no financial or proprietary interests in any material
discussed in this article.
\bibliographystyle{plain}
\bibliography{biblioMetaQF}
\end{document}